\documentclass[11pt]{article}
\usepackage{amsmath,amsthm}
\usepackage{mathrsfs}
\usepackage{amssymb}
\usepackage{amsfonts}
\usepackage{bm}
\usepackage{mathtools}
\usepackage{tikz}
\usetikzlibrary{calc,shapes, backgrounds}
\pgfdeclarelayer{edgelayer}
\pgfdeclarelayer{nodelayer}
\pgfsetlayers{edgelayer,nodelayer,main}

\usetikzlibrary{calc,shapes, backgrounds, arrows,positioning}
\usepackage{xcolor}
\usepackage{geometry}
\usetikzlibrary{decorations.markings}
\usetikzlibrary{calc}
\tikzset{directed/.style={decoration={markings,
mark=at position .7 with {\arrow[scale=1.4]{stealth}}},
postaction={decorate}}}
\usepackage[pdfstartview=FitH,
bookmarksopen=false,
colorlinks=true,
linkcolor=blue,
citecolor=blue,
urlcolor=blue]{hyperref}

\makeatletter
\def\@seccntDot{.}
\def\@seccntformat#1{\csname the#1\endcsname\@seccntDot\hskip 0.5em}
\renewcommand\section{\@startsection{section}{1}{\z@}%
{18\p@ \@plus 6\p@ \@minus 3\p@}%
{9\p@ \@plus 6\p@ \@minus 3\p@}%
{\large\bfseries\boldmath}}
\renewcommand\subsection{\@startsection{subsection}{2}{\z@}%
{12\p@ \@plus 6\p@ \@minus 3\p@}%
{3\p@ \@plus 6\p@ \@minus 3\p@}%
{\bfseries\boldmath}}
\renewcommand\subsubsection{\@startsection{subsubsection}{3}{\z@}%
{12\p@ \@plus 6\p@ \@minus 3\p@}%
{\p@}%
{\bfseries\boldmath}}
\makeatother

\usepackage{microtype}

\theoremstyle{plain}
\newtheorem{theorem}{Theorem}
\newtheorem{theoremtwo}{Theorem}
\newtheorem*{theorem*}{Theorem}
\newtheorem{lemma}{Lemma}

\numberwithin{equation}{section}
\allowdisplaybreaks
\begin{document}

\title{The game of Cops and Robbers on directed graphs with forbidden subgraphs}

\author{Mingrui Liu
\thanks{Yau Mathematical Sciences Center, Tsinghua University, Beijing 100084, China
({\tt lmr16@mails.tsinghua.edu.cn}).}
}

\maketitle

\begin{abstract}
\par\vspace{2mm}
The traditional game of cops and robbers is played on undirected graph. Recently, the same game played on directed graph is getting attention by more and more people.
We  knew that if we forbid some subgraph we can bound the cop number of the corresponding class of graphs. In this paper, we analyze the game of cops and robbers on $\Vec{H}$-free digraphs. However, it is not the same as the case of undirected graph. So we give a new concept ($\Vec{H}^*$-free) to get a similar conclusion about the case of undirected graph.



\end{abstract}

\section{Introduction}
The game of $\textit{Cops and robbers}$ first introduced by Nowakowski and Winkler \cite{MR685631} and independently by Quilliot \cite{name1} is a two player game played on a graph $G = (V, E)$. The first player, called the $\textit{cop player}$ has $k$ ($k\geq 1$) pieces and the second player, the $\textit{robber}$, has only one piece. In the beginning, the cop player places his $k$ pieces on some vertices (not necessarily distinct) of the graph and the robber places his piece on one vertex of the same graph. In the first round, the cop player can move all his pieces before the robber. Each piece of the cop can be moved to an adjacent vertex or stay idle. In the robber's move, he can also move his piece to an adjacent vertex or place it still. After the cop player and the robber finish their moves, one round is finished. It's the time for the next round, the two players do the same thing like the previous round alternately. The game ends when some piece of the cop and the robber are on the same vertex (that is, the cops catch the robber). In this case, we say the cop player wins or cops win. The robber wins if he can never be caught by the cops. Both players have complete information, which means they know the graph and the positions of all the pieces.

Nowakowski and Winkler \cite{MR685631} and Quilliot \cite{name1} considered the case for $k=1$. That means one cop catches one robber. Later, Aigner and Fromme \cite{MR739593} generalized the game to several cops. The key problem for this game is to know how many cops are needed to catch the robber. We call the minimum integer of cops required to capture the robber the $\textit{cop number}$ and denoted by $c(G)$.  Many mathematicians have already done some research for the cop number. Frankl \cite{MR890640} conjectured that for any connected $n$-vertex graph $G$ it holds that $c(G)=O(\sqrt{n})$ in 1987. This conjecture, known as Meyniel's conjecture still remains open. The best known upper bound, provided independently by \cite{MR2965383, MR2979296, MR2837608}, said that the cop number of any graph on $n$ vertices is upper bounded by $n2^{-(1+o(1))\sqrt{n}}$. That means we even cannot give a loose upper bound $O(n^{1-o(1)})$ so far. Aigner and Fromme \cite{MR739593} proved that for any planar graph $G$, $c(G)\leq 3$ in 1984. Quilliot \cite{MR782627} gave an upper bound $2g+3$ about the cop number for any graph $G$ of genus $g$ in 1985. Schroeder \cite{Schroeder} improved this bound to $\lfloor \frac{3g}{2}\rfloor +3$ and conjectured that $c(G) \leq g+3$ for any graph of genus $g$ in 2001.

In \cite{MR739593}, a simple forbidden subgraph condition was given to ensure that $c(G)\geq\delta(G)$. Joret, Kami\'{n}ski and Theis \cite{MR2791289} gave an important theorem about the relationship between the class of $H$-free graphs and bounded cop number. More recently, the game of cops and robber has been considered on directed graphs, or digraphs, for short. In \cite{MR919881}, upper bounds were provided for directed abelian Cayley graphs. In \cite{MR2979296}, Frieze, Krivelevich and Loh proved that $c(\Vec{D}) = O(n\frac{(loglogn)^2}{logn})$ for any strongly connected digraph $\Vec{D}$. For more research about the game of cops and robbers on directed graph, please refer to \cite{oriented}.

In this paper, we discuss the relationship between the class of $\Vec{H}$-free digraphs and the bounded cop number. We find that we cannot find a subgraph $\Vec{H}$ such that the class of $\Vec{H}$-free digraphs has bounded cop number, even if the directed graphs are strongly connected.

\begin{theorem}
For any  directed graph $\Vec{H}$, the class of $\Vec{H}$-free digraphs has unbounded cop number.
\end{theorem}

\begin{theorem}
For any  directed graph $\Vec{H}$, the class of $\Vec{H}$-free
strongly connected digraphs has unbounded cop number.
\end{theorem}

That means we cannot find any given structures to guarantee that $\Vec{H}$-free digraphs have bounded cop number. Maybe the condition of $\Vec{H}$-free is too strict and we find another way to forbid the subgraphs and we call it $\Vec{H}^*$-free. We get one conclusion about the $\Vec{H}^*$-free and the bounded cop number.

\begin{theorem}
Let $\Vec{D}$ be a $\Vec{P_k}^*$-free strongly connected digraph for $k\geq 3$, then $c(\Vec{D})\leq k-2$.
\end{theorem}

\section{Preliminaries}
When we mention a graph $G$, we mean that $G$ is simple, finite, connected and undirected. When we mention a digraph $\Vec{D}$, we mean that $\Vec{D}$ has no loops or parallel arcs (arcs with the same head and the same tail), and is weakly connected (its underlying undirected graph is connected). The oriented graph  is a directed graph  without $2$-cycles, which means its underlying undirected graph has no multiple edges.

Let $\Vec{H}$ be a digraph. A digraph is called $\Vec{H}$-free if it does not contain $\Vec{H}$ as an induced subgraph. The adjacency matrix of $\Vec{H}$ of order $n$ is the $n\times n$ matrix $ (a_{uv})$, where $a_{uv}=1$ or $0$ according to whether there exists an arc in $\Vec{H}$ with tail $u$ and head $v$.
Let $\Vec{P_k}$ be an oriented graph of a path of order $k$.
A digraph $\Vec{D}$ is called $\Vec{P_k}^*$-free if  $\Vec{D}$ does not contain any induced subgraph whose upper triangular part of its adjacency matrix is the same as $\Vec{P_k}$.

If a digraph $\Vec{D}$ does not contain $\Vec{P_k}$ as a subgraph (not necessary an induced subgraph), then $\Vec{D}$ is $\Vec{P_k}^*$-free. If $\Vec{D}$ is $\Vec{P_k}^*$-free, then $\Vec{D}$ must be $\Vec{P_k}$-free.

We say that the cop number is $\textit{bounded}$ for a class of digraphs, if there exits a constant $C$ such that the cop number of every digraph in this class is at most $C$; otherwise the cop number is $\textit{unbounded}$ for this class.

Let $\Vec{D}$ be a digraph. The neighbourhood of a vertex $v$ in $\Vec{D}$, denoted by $N(v)$, is the set of vertices adjacent to $v$ in its underlying undirected graph. Let  $N^-(v)$ (resp. $N^+(v)$) be the set of vertices who only has arc with head $v$ (resp. with head $v$) and $N^\pm(v)$ be the set of vertices who has both arcs with head $v$ and with tail $v$ in $\Vec{D}$. Then $N^-(v)$, $ N^+(v)$, $ N^\pm(v)$ are pairwise disjoint and $N(v)=N^-(v)\cup N^+(v)\cup N^\pm(v)$. Let $N^-(v)=\{x^-_1,\ldots,x^-_{i_1}\}$, $N^+(v)=\{x^+_1,\ldots,x^+_{i_2}\}$ and $N^\pm(v)=\{x^\pm_1,\ldots,x^\pm_{i_3}\}$.
A $\textit{clique substitution}$ at a vertex $v$ consists in replacing $v$ with new vertices $\{y^-_1,\ldots,y^-_{i_1}\}\cup \{y^+_1,\ldots,y^+_{i_2}\}\cup  \{y^\pm_1,\ldots,y^\pm_{i_3}\}$ satisfies (i) there is one arc with tail $x^-_i$ and head $y^-_i$, $1\le i\le i_1$, there is one arc with tail $y^+_i$ and head $x^+_i$, $1\le i\le i_2$ and there are two opposite arcs between $x^\pm_i$ and $y^\pm_i$, $1\le i\le i_3$; (ii) $\{y^-_1,\ldots,y^-_{i_1}\}$, $\{y^+_1,\ldots,y^+_{i_2}\}$ and $  \{y^\pm_1,\ldots,y^\pm_{i_3}\}$ are three cliques with two opposite arcs, respectively; (iii) there are two opposite arcs between $y^\pm_i$ for $1\le i\le i_3$ and all $y\in \{y^-_1,\ldots,y^-_{i_1}\}\cup \{y^+_1,\ldots,y^+_{i_2}\}$ and there is one arc with tail $y'$ and head $y''$ for all $y'\in \{y^-_1,\ldots,y^-_{i_1}\}$ and $y''\in \{y^+_1,\ldots,y^+_{i_2}\}$ (see Figure 1 for example).  The digraph obtained from a digraph $\Vec{D}$ by substituting a clique at each vertex of $\Vec{D}$ will be denoted by $\Vec{D}^+$.
An $\textit{arc substitution}$ at an arc $a$ is to replace $a$ by a same directed path of any length (see Figure 2 for example).


We can find that a strongly connected directed graph is still strongly connected after the operation of clique substitution or arc substitution.

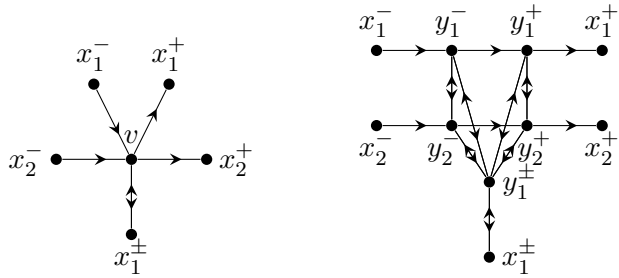
\begin{figure}[hbt]
\begin{center}
\tikzstyle{vertex}=[circle,fill=black,inner sep=1.5pt]
\tikzstyle{new}=[circle,fill=red,inner sep=1.5pt]
\tikzstyle{blue}=[circle,fill=blue, inner sep=1.5pt]
\tikzstyle{green}=[circle,fill=green, inner sep=1.5pt]

\tikzstyle{curly edge}=[]
\tikzstyle{straight edge}=[]

\begin{tikzpicture}
	\begin{pgfonlayer}{nodelayer}
		\node [style=vertex] (0) at (-0.5, 1) {};
		\node [style=vertex] (1) at (-1, 0) {};
		\node [style=vertex] (2) at (0, 0) {};
		\node [style=vertex] (3) at (0, -1) {};
		\node [style=vertex] (4) at (1, 0) {};
		\node [style=vertex] (5) at (0.5, 1) {};
		\node at (0, 0.3) {$v$};
		\node at (-0.5, 1.4) {$x^-_1$};
		\node at (-1.4, 0) {$x^-_2$};
		\node at (0, -1.3) {$x^\pm_1$};
		\node at (1.4, 0) {$x^+_2$};
		\node at (0.5, 1.4) {$x^+_1$};
		
	\end{pgfonlayer}
	\begin{pgfonlayer}{edgelayer}
		\draw [directed] (0) to (2);
		\draw [directed] (1) to (2);
		\draw [directed] (2) to (4);
		\draw [directed] (2) to (5);
		\draw [directed] (2) to (3);
		\draw [directed] (3) to (2);
	\end{pgfonlayer}
\end{tikzpicture}
\hspace{1cm}
\begin{tikzpicture}
	\begin{pgfonlayer}{nodelayer}
		\node [style=vertex] (0) at (-1.5, 1) {};
		\node [style=vertex] (1) at (-0.5, 1) {};
		\node [style=vertex] (2) at (0.5, 1) {};
		\node [style=vertex] (3) at (1.5, 1) {};
		\node [style=vertex] (4) at (-1.5, 0) {};
		\node [style=vertex] (5) at (-0.5, 0) {};
		\node [style=vertex] (6) at (0.5, 0) {};
		\node [style=vertex] (7) at (1.5, 0) {};
		\node [style=vertex] (8) at (0, -0.75) {};
		\node [style=vertex] (9) at (0, -1.75) {};
		\node at (-1.5, 1.4) {$x^-_1$};
		\node at (-.5, 1.4) {$y^-_1$};
		\node at (.5, 1.4) {$y^+_1$};
		\node at (1.5, 1.4) {$x^+_1$};
		\node at (-1.5, -.3) {$x^-_2$};
		\node at (-.6, -.3) {$y^-_2$};
		\node at (.6, -.3) {$y^+_2$};
		\node at (1.5, -.3) {$x^+_2$};
		\node at (0.4, -.75) {$y^\pm_1$};
		\node at (0.4, -1.75) {$x^\pm_1$};
	
	\end{pgfonlayer}
	\begin{pgfonlayer}{edgelayer}
		\draw [directed] (0) to (1);
		\draw [directed] (1) to (2);
		\draw [directed] (2) to (3);
		\draw [directed] (4) to (5);
		\draw [directed] (5) to (6);
		\draw [directed] (6) to (7);
		\draw [directed] (8) to (5);
		\draw [directed] (8) to (9);
		\draw [directed] (5) to (1);
		\draw [directed] (2) to (6);
		\draw [directed] (6) to (2);
		\draw [directed] (6) to (8);
		\draw [directed] (8) to (6);
		\draw [directed] (1) to (8);
		\draw [directed] (8) to (1);
		\draw [directed] (2) to (8);
		\draw [directed] (8) to (2);
		\draw [directed] (5) to (8);
		\draw [directed] (9) to (8);
		\draw [directed] (1) to (5);
		\end{pgfonlayer}
\end{tikzpicture}
\label{fig:graphs}
\caption{Example of clique substitution at vertex $v$.}
\end{center}
\end{figure}

\begin{figure}[hbt]
\begin{center}
\tikzstyle{vertex}=[circle,fill=black,inner sep=1.5pt]
\tikzstyle{new}=[circle,fill=red,inner sep=1.5pt]
\tikzstyle{blue}=[circle,fill=blue, inner sep=1.5pt]
\tikzstyle{green}=[circle,fill=green, inner sep=1.5pt]

\tikzstyle{curly edge}=[]
\tikzstyle{straight edge}=[]

\begin{tikzpicture}
	\begin{pgfonlayer}{nodelayer}
		\node [style=vertex] (0) at (-1, 0) {};
		\node [style=vertex] (1) at (0, 0) {};
		\node at (-0.5,0.3) {$a$};
	\end{pgfonlayer}
	\begin{pgfonlayer}{edgelayer}
		\draw [directed] (0) to (1);
	\end{pgfonlayer}
\end{tikzpicture}
\hspace{1cm}
\begin{tikzpicture}
	\begin{pgfonlayer}{nodelayer}
		\node [style=vertex] (0) at (-1, 0) {};
		\node [style=vertex] (1) at (-0.5, 0) {};
		\node [style=vertex] (2) at (0, 0) {};
		\node [style=vertex] (3) at (0.5, 0) {};
	\end{pgfonlayer}
	\begin{pgfonlayer}{edgelayer}
		\draw [directed] (0) to (1);
		\draw [directed] (1) to (2);
		\draw [directed] (2) to (3);
		\end{pgfonlayer}
\end{tikzpicture}
\label{fig:graphs}
\caption{Example of arc substitution at arc $a$.}
\end{center}
\end{figure}
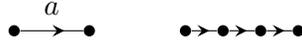

\section{Forbidden Induced Subgraph}
\begin{theoremtwo}
For any directed graph $\Vec{H}$, the class of $\Vec{H}$-free digraphs has unbounded cop number.
\end{theoremtwo}
\begin{proof}
Suppose not, let $\Vec{H}$ be a directed graph such that the class of $\Vec{H}$-free graphs has bounded cop number. First, suppose the underlying graph of $\Vec{H}$ contains a cycle. We know that the class of directed graphs whose underlying graphs are trees is contained in the class of $\Vec{H}$-free graphs. However the class of directed graphs whose underlying graphs are trees is cop-unbounded since the cop number for any digraph $\Vec{D}$ is at least as large as the number of its sources, a contradiction. Now we consider the case that the underlying graph of $\Vec{H}$ does not contain a cycle. Then the underlying graph of $\Vec{H}$ must be a forest.
Since $\Vec{H}$ is a forest, $\Vec{H}$ must contain an induced arc  $\Vec{P_2}$. However, we can find a class of $\Vec{P_2}$-free graphs which is cop-unbounded. In fact we just replace each edge in the incidence graphs of finite projective plane (the cop number is $\sqrt{n}$) \cite{MR2985811} as two oppositely oriented arcs with the same ends, a contradiction. So the class of $\Vec{H}$-free digraphs has unbounded cop number.
\end{proof}

We known that the cop number of any directed graph can reduce to the cop number of its strongly connected components. So we restrict the directed graph to the strongly connected directed graph. However, we still cannot get the desired conclusion. First, we need four lemmas.

\begin{lemma}
For any directed graph, clique substitution dose not decrease the cop number.
\end{lemma}

\begin{proof}
Let $\Vec{D}$ be a directed graph and clique substitution at every vertex in $\Vec{D}$ is denoted by $\Vec{D}^+$. To each vertex $v$ $\in$ $V(\Vec{D})$ there corresponds a clique in $\Vec{D}^+$, which we denote by $\phi(v)$. We simultaneously play two games: one on $\Vec{D}$ and another on $\Vec{D}^+$. We assume that we have a winning strategy for the cop player on $\Vec{D}^+$ and we simulate her moves on $\Vec{D}$.  On
the other hand, the robber is playing on $\Vec{D}$ and we simulate his moves on $\Vec{D}^+$.

At the beginning, according to the strategy that the cops are placed on $\Vec{D}^+$, the corresponding cops in $\Vec{D}$ are placed in the obvious way: if a cop in $\Vec{D}^+$ is on a vertex of  $\phi(v)$ for  $v$ $\in V(\Vec{D})$, then the corresponding cop in $\Vec{D}$ is put on $v$. Then, we put the robber in $\Vec{D}^+$ on an arbitrary vertex of the clique $\phi(u)$ if the robber is in $u$ in $\Vec{D}$.

In the first round, we let all the cops in $\Vec{D}$ stay idle and the corresponding cops in $\Vec{D}^+$ also stay idle. Then it's the robber's turn to move in $\Vec{D}$. The robber has two choices: stay idle or move. If the robber stays idle, then we do not move the robber in $\Vec{D}^+$. Assume the robber moves in $\Vec{D}$, for example, from $u$ to $v$. Let $u'v'$ be the (unique) arc in $\Vec{D}^+$, where $u'\in \phi(u)$ and $v'\in \phi(v)$. We will consider two cases.

{\bf Case 1} In $\Vec{D}^+$, the robber is on $u'$.

In this case, we move the robber to $v'$ (the arc works by the definition of clique substitution) and the cops in $\Vec{D}^+$ have one round to move by winning strategy. Assume one of cops, say $c_1$, moves from $x'$ to $y'$, where $x'\in \phi(x)$ and $y'\in \phi(y)$. If $x=y$, then $c_1$ does nothing in $\Vec{D}$; otherwise $c_1$ moves from $x$ to $y$.

{\bf Case 2} In $\Vec{D}^+$, the robber is on another vertex of $\phi(u)$, say $u''$.

In this case, we move the robber first from $u''$ to $u'$ by one step (it works by the definition of clique substitution). Then the cops in $\Vec{D}^+$  have one round to move by winning strategy. Assume one of cops, say $c_1$, moves from $x'$ to $y'$, where $x'\in \phi(x)$ and $y'\in \phi(y)$. In the next round in $\Vec{D}^+$, we can move the robber from $u'$ to $v'$ and the cops can also move for another round by winning strategy. In this round, we assume $c_1$ moves from $y'$ to $z'$, where $z'\in \phi(z)$. By the definition of clique substitution, we have $x=y=z$ or $x=y\not=z$ or $x\not=y=z$ in $\Vec{D}$. If $x=y=z$, then $c_1$ does nothing in $\Vec{D}$; otherwise $c_1$ moves from $x$ to $z$.


Now we get a strategy for the cops in $\Vec{D}$. By our assumption, the robber will be caught in $\Vec{D}^+$. That means at least one cop and the robber are on the clique $\phi(v)$ for some vertex $v\in V(\Vec{D})$ and the corresponding cop and the robber in $\Vec{D}$ will be on the same vertex by our translation.
\end{proof}

\begin{lemma}
For any directed graph, arc substitution does not decrease the cop number.
\end{lemma}
\begin{proof}
Let $\Vec{D}^{++}$ be a digraph obtained from $\Vec{D}$ by subdividing all arcs of $\Vec{D}$ at the same time. 
We can construct two maps $\iota$: $V(\Vec{D})$$\rightarrow$ $V(\Vec{D}^{++})$ and $\psi$: $V(\Vec{D}^{++})\rightarrow V(\Vec{D})$ such that
$\iota$: $V(\Vec{D})$$\rightarrow$ $V(\Vec{D}^{++})$ is the natural map so that the vertices in $V(\Vec{D}^{++})\backslash\iota[V(\Vec{D})]$ are these new vertices added to $\Vec{D}$ to obtain $\Vec{D}^{++}$. Suppose that $uv$ is an arc in $\Vec{D}$, then there is a directed path $P_{u,v}$ in $\Vec{D}^{++}$ joining $\iota(u)$ with $\iota(v)$. The map $\psi$ sends the vertices of $P_{u,v}$ except $\iota(u)$ to the vertex $v\in V(\Vec{D})$, and vertex $\iota(u)$ to the vertex $u\in V(\Vec{D})$.

Now we know that the corresponding relationship between $V(\Vec{D})$ and $V(\Vec{D}^{++})$. We use the "copy" strategy to copy the winning strategy in $\Vec{D}^{++}$ and will catch the robber in digraph $\Vec{D}$ just like Lemma 1.
\end{proof}

Now we know some powerful lemmas and we can explore the relationship between the class of $\Vec{H}$-free strongly connected directed graphs and the bounded cop number.

\begin{lemma}
The class of $\Vec{H}$-free strongly connected directed graphs has unbounded cop number, where $\Vec{H}$ is one of the four structures in Figure 3.
\end{lemma}

\begin{figure}[hbt]
\begin{center}

\tikzstyle{vertex}=[circle,fill=black,inner sep=1.5pt]
\tikzstyle{new}=[circle,fill=red,inner sep=1.5pt]
\tikzstyle{blue}=[circle,fill=blue, inner sep=1.5pt]
\tikzstyle{green}=[circle,fill=green, inner sep=1.5pt]

\tikzstyle{curly edge}=[]
\tikzstyle{straight edge}=[]

\begin{tikzpicture}
	\begin{pgfonlayer}{nodelayer}
	    \node [style=vertex] (0) at (0, 0) {};
		\node [style=vertex] (1) at (0, 1) {};
		\node [style=vertex] (2) at (-0.5, -0.75) {};
		\node [style=vertex] (3) at (0.5, -0.75) {};
	\end{pgfonlayer}
	\begin{pgfonlayer}{edgelayer}
		\draw [directed] (0) to (1);
		\draw [directed] (0) to (2);
		\draw [directed] (0) to (3);
	\end{pgfonlayer}
\end{tikzpicture}
\hspace{1cm}
\begin{tikzpicture}
	\begin{pgfonlayer}{nodelayer}
	    \node [style=vertex] (0) at (0, 0) {};
		\node [style=vertex] (1) at (0, 1) {};
		\node [style=vertex] (2) at (-0.5, -0.75) {};
		\node [style=vertex] (3) at (0.5, -0.75) {};
	\end{pgfonlayer}
	\begin{pgfonlayer}{edgelayer}
		\draw [directed] (1) to (0);
		\draw [directed] (0) to (2);
		\draw [directed] (0) to (3);
	\end{pgfonlayer}
\end{tikzpicture}
\hspace{1cm}
\begin{tikzpicture}
	\begin{pgfonlayer}{nodelayer}
	    \node [style=vertex] (0) at (0, 0) {};
		\node [style=vertex] (1) at (0, 1) {};
		\node [style=vertex] (2) at (-0.5, -0.75) {};
		\node [style=vertex] (3) at (0.5, -0.75) {};
	\end{pgfonlayer}
	\begin{pgfonlayer}{edgelayer}
		\draw [directed] (1) to (0);
		\draw [directed] (2) to (0);
		\draw [directed] (3) to (0);
	\end{pgfonlayer}
\end{tikzpicture}
\hspace{1cm}
\begin{tikzpicture}
	\begin{pgfonlayer}{nodelayer}
	    \node [style=vertex] (0) at (0, 0) {};
		\node [style=vertex] (1) at (0, 1) {};
		\node [style=vertex] (2) at (-0.5, -0.75) {};
		\node [style=vertex] (3) at (0.5, -0.75) {};
	\end{pgfonlayer}
	\begin{pgfonlayer}{edgelayer}
		\draw [directed] (0) to (1);
		\draw [directed] (2) to (0);
		\draw [directed] (3) to (0);
	\end{pgfonlayer}
\end{tikzpicture}

\end{center}

\label{fig:graphs}
\caption{Four structures of $\Vec{H}$ mentioned in Lemma 3.}
\end{figure}
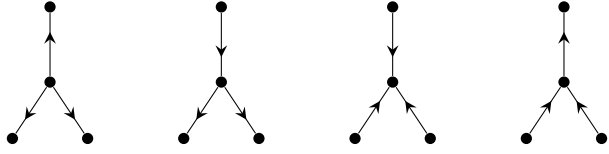

\begin{proof}
Let $\mathcal{D}$ be any class of strongly connected digraphs with unbounded cop number and $\mathcal{D}^+$ := $\{\Vec{D}^+| \Vec{D}\in\mathcal{D}\}$. Notice that all digraphs in $\mathcal{D}^+$ are also strongly connected and $\Vec{H}$-free. Applying Lemma 1, we draw a conclusion that the cop number of digraphs in $\mathcal{D}^+$ is unbounded.
\end{proof}

\begin{lemma}
The class of strongly connected digraphs with undirected girth at least $l$ ($l\geq 2$) has unbounded cop number.
\end{lemma}
\begin{proof}
Let $\mathcal{D}$ be any class of strongly connected digraphs with unbounded cop number and $\mathcal{D}^{++}$ := $\{\Vec{D}^{++}| D\in\mathcal{D}\}$, $\Vec{D}^{++}$ is obtained by subdividing arcs of $\Vec{D}$ sufficient times to make the undirected girth of $\Vec{D}^{++}$ at least $l$. The digraphs in $\mathcal{D}^{++}$ are also strongly connected. Applying Lemma 2, we find that the cop number of digraphs in $\mathcal{D}^{++}$ is unbounded.
\end{proof}

Now we are ready to complete the proof of Theorem 2.

\begin{theoremtwo}
For any  directed graph $\Vec{H}$, the class of $\Vec{H}$-free
strongly connected digraphs has unbounded cop number.
\end{theoremtwo}
\begin{proof}
Let $\Vec{H}$ be a directed graph such that the class of $\Vec{H}$-free strongly connected digraphs has bounded cop number. By Lemma 4, we have that the underlying graph of $\Vec{H}$ must be a forest.

Now suppose that the underlying graph of $\Vec{H}$ is not a path, which means that $\Vec{H}$ must contains at least one of four structures mentioned in Lemma 3 as an induced subgraph. However, we know that's also impossible by Lemma 3. So the underlying graph of $\Vec{H}$ must be a path, which must contain $\Vec{P_2}$ and makes a contradiction like the proof of Theorem 1.
\end{proof}

\section{$\Vec{P_k}^*$-free}
We now turn our attention to the class of  $\Vec{P_k}^*$-free graphs. We can get the following theorem.
\begin{theoremtwo}
Let $\Vec{D}$ be a $\Vec{P_k}^*$-free strongly connected digraph for $k\geq 3$, then $c(\Vec{D})\leq k-2$.
\end{theoremtwo}
\begin{proof}
Note that $\Vec{D}$ does not contain $\Vec{P_k}$ as an induced subgraph. Moreover, $\Vec{D}$ also does not contain any induced subgraphs which has the same upper triangular part of adjacency matrix as $\Vec{P_k}$.
\begin{figure}[hbt]
\begin{center}

\tikzstyle{vertex}=[circle,fill=black,inner sep=1.5pt]
\tikzstyle{new}=[circle,fill=red,inner sep=1.5pt]
\tikzstyle{blue}=[circle,fill=blue, inner sep=1.5pt]
\tikzstyle{green}=[circle,fill=green, inner sep=1.5pt]

\tikzstyle{curly edge}=[]
\tikzstyle{straight edge}=[]

\begin{tikzpicture}
	\begin{pgfonlayer}{nodelayer}
	    \node [style=vertex] (0) at (-2, 0) {};
		\node [style=vertex] (1) at (-1, 0) {};
		\node [style=vertex] (2) at (0, 0) {};
		\node [style=vertex] (3) at (1, 0) {};
		\node [style=vertex] (4) at (2, 0) {};
	\end{pgfonlayer}
	\begin{pgfonlayer}{edgelayer}
		\draw [directed] (0) to (1);
		\draw [directed] (1) to (2);
		\draw [dashed] (2) to (3);
		\draw [directed] (3) to (4);
		\draw [directed] (2,0) arc(0:180:1cm and 1cm);
	\end{pgfonlayer}
\end{tikzpicture}
\hspace{1cm}
\begin{tikzpicture}
    \begin{pgfonlayer}{nodelayer}
	    \node [style=vertex] (0) at (-2, 0) {};
		\node [style=vertex] (1) at (-1, 0) {};
		\node [style=vertex] (2) at (0, 0) {};
		\node [style=vertex] (3) at (1, 0) {};
		\node [style=vertex] (4) at (2, 0) {};
	\end{pgfonlayer}
	\begin{pgfonlayer}{edgelayer}
		\draw [directed] (0) to (1);
		\draw [directed] (1) to (2);
		\draw [dashed] (2) to (3);
		\draw [directed] (3) to (4);
		\draw [directed] (2,0) arc(0:180:1cm and 1cm);
		\draw [directed] (2,0) arc(0:180:2cm and 1cm);
	\end{pgfonlayer}
\end{tikzpicture}
\hspace{1cm}
\begin{tikzpicture}
	\begin{pgfonlayer}{nodelayer}
	    \node [style=vertex] (0) at (-2, 0) {};
		\node [style=vertex] (1) at (-1, 0) {};
		\node [style=vertex] (2) at (0, 0) {};
		\node [style=vertex] (3) at (1, 0) {};
		\node [style=vertex] (4) at (2, 0) {};
	\end{pgfonlayer}
	\begin{pgfonlayer}{edgelayer}
		\draw [directed] (0) to (1);
		\draw [directed] (1) to (2);
		\draw [dashed] (2) to (3);
		\draw [directed] (3) to (4);
		\draw [directed] (1,0) arc(0:180:0.5cm and 1cm);
		\draw [directed] (1,0) arc(0:180:1cm and 1cm);
	\end{pgfonlayer}
\end{tikzpicture}

\end{center}

\label{fig:graphs}
\caption{Some examples who have the same upper triangular part of adjacency matrix as $\Vec{P_k}$ .}
\end{figure}
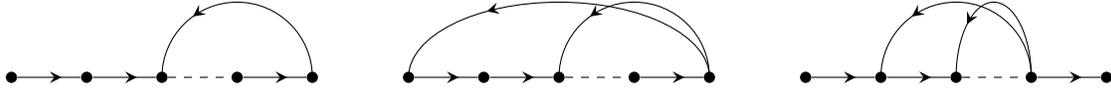

Suppose that we can never use $k-2$ cops to catch the robber, that means there exists a minimal distance $t$ ($t\geq 1$) between cops and robber in unlimited rounds of the game. Even if the cops cannot catch the robber (at least one cop and the robber are on the same vertex), the cops also have a strategy such that at least one cop has a distance of $t$ from the robber after finite rounds based on our assumption. Suppose a cop $c_1$ is on the vertex $u$ and the robber is on the vertex $v$, there is a shortest directed path $\Vec{P}_{t+1}$: $u-u_1-\ldots-u_{t-1}-v$. Next, the robber will move to a new vertex $v_1$ to escape capture and $c_1$ will move to $u_1$ to decrease the distance between the cops and the robber. The robber uses his strategy to escape and has a walk: $v_1-v_2-\ldots$ (perhaps $v_{i_1}$ and $v_{i_2}$ are the same; $v_{i_3}$ is $u_i$), $c_1$ also has a walk: $u_1-\ldots-u_{t+1}-v_1-v_2-\ldots$. Let the cop $c_1$ move along the path $\Vec{P}_{t+1}$ and then just repeat the robber's former moves. Since the digraph $\Vec{D}$ is finite so we can let the other $k-3$ cops will be on the same vertex with $c_1$ after finite rounds. Next, we just consider the case all $k-2$ cops are on one vertex and have a minimal distance $t$ ($t\geq 1$) from the robber.

Without loss of generality, suppose $k-2$ cops are on the vertex $u$ and the robber is on the vertex $v$, there is a shortest directed path $\Vec{P}_{t+1}$: $u-u_1-\ldots-u_{t-1}-v$ between $u$ and $v$. The robber moves to a new vertex $v_1$ to escape capture. There is no arc from $u$, $u_1$, $\ldots$, $u_{t-1}$ to $v_1$, suppose not, the distance between cops and robber will be less than $t$. We leave one cop $c_1$ on $u$ and let other $k-3$ cops move to $u_1$. The robber cannot choose any existing vertex to move because this move will decrease the minimal distance between cops and robber. The robber must choose another new vertex $v_2$ to escape. We leave another cop $c_2$ on $u_1$ and let other $k-4$ cops move to $u_2$. The robber use his escaping strategy and has his new path: $v_1-\ldots-v_{k-2}$, $v_i$ and $v_j$ are different and $v_i$ and $u_i$ are also different based on our analysis. The cops will leave one cop still and others go on catching again and again. Finally, we can get a directed path of length $k+t-2$: $u-u_1-\ldots-u_{t+1}-v-v_1-\ldots-v_{k-2}$. Since $t\geq 1$, we get a directed path $P_k$. This path cannot have any arc from the previous vertex to the later vertex by our construction, so $\Vec{D}$ contains an induced subgraph who has the same upper triangular part of adjacency matrix as $\Vec{P}_k$. However, by our assumption, $\Vec{D}$ is $\Vec{P_k}^*$-free. We get a contradiction.

\end{proof}

We knew that we have the similar conclusion ($k-2$ cops can catch the robber for any $P_k$-free undirected graph) in \cite{MR2791289}. So we provide a new perspective about how to forbid the subgraphs and get a similar conclusion in directed graphs.

\end{document}